\documentclass[11pt]{article}
\usepackage{amsmath, amssymb, amsthm, epsfig, verbatim, xcolor}

\newtheorem{theorem}{Theorem}[section]
\newtheorem{lemma}[theorem]{Lemma}
\newtheorem{prop}[theorem]{Proposition}

\newtheorem{example}{Example}[section]

\newcommand{\gaussm}[3]{\genfrac{[}{]}{0pt}{}{#1}{#2}_{#3}}
\newcommand{\supp}{\operatorname{supp}}

\newcommand{\wt}{\operatorname{wt}}

\begin{document}

\title{\textbf{On the minimum number of minimal codewords}}
\author{Romar dela Cruz\textsuperscript{1}, Michael Kiermaier\textsuperscript{2}, Sascha Kurz\textsuperscript{2} and Alfred Wassermann\textsuperscript{2}\\\\
\small{\textsuperscript{1}Institute of Mathematics, University of the Philippines Diliman, Philippines}\\
\small{\textsuperscript{2}Department of Mathematics, University of Bayreuth, Germany}}
\date{}

\maketitle

\begin{abstract}
We study the minimum number of minimal codewords in linear codes from the point of view of projective geometry. We derive bounds and in some cases determine the exact values.  We also present an extension to minimal subcode supports. 
\end{abstract}

\bigskip

\section{Introduction}
The support of a vector is the set of its nonzero coordinate positions.  In a linear code, a nonzero codeword is said to be minimal if its support is minimal (with respect to set inclusion).  Minimal codewords were first studied in connection with decoding \cite{Hwang, Agrell, Agrell2}.  They were reintroduced by Massey in the context of secret sharing \cite{Massey} and they were used in a protocol for secure two-party computation \cite{CCP}. Minimal codewords can be viewed as circuits in matroids and also as cycles in graphs. 

In general, it is difficult to determine the set of minimal codewords of a given linear code.  This was only done for some classes of codes, for instance see \cite{Agrell, AB, BM, DKL, YF, SST}.  The authors in \cite{Maxmin, Maxmin2, Minmin} investigated the maximum and minimum number of minimal codewords in binary linear codes.  Given the length and dimension, bounds and some exact values were presented.  This can be seen as a coding-theoretic analogue of problems considered in the setting of matroids \cite{DSL} and graphs \cite{ES}.  

In this work, we continue the study of the minimal codewords using techniques from projective geometry. First, we present a geometric characterization of minimal (and non-minimal) codewords.  Then we use this characterization to derive a lower bound on the number of minimal codewords of a linear code.  As a consequence, we obtain exact values of the minimum number of minimal codewords of linear codes of certain length and dimension.  Our result applies to both binary and non-binary linear codes. The geometric approach can also be extended to minimal subcode supports.  

\section{Theoretical background}

Let $\mathbb{F}_q$ be the finite field with $q$ elements where $q$ is a power of a prime. A $q$-ary $[n,k]_q$ \emph{linear code} $C$ is a $k$-dimensional subspace of the $n$-dimensional vector space $\mathbb{F}_q^n$. Elements $c\in C$ are called \emph{codewords} and $n$ is called the \emph{length} of the code. The \emph{support} of a codeword $c$ is the set of coordinates with a non-zero entry, i.e., $\supp(c)=\left\{i\in\{1,\dots,n\}\,:\, c_i\neq 0\right\}$. The \emph{Hamming weight} $\wt(c)$ of a codeword is the cardinality $|\supp(c)|$ of its support. We define $\supp(C)=\cup_{c\in C}\supp(c)$ and call $|\supp(C)|$ the \emph{effective length} of $C$. We call a code $C$ \emph{non-trivial} if its dimension $\dim(C)=k$ is at least $1$. Here we assume that all codes are non-trivial and that the effective length equals the length $n$ (or $n(C)$ to be more precise). A matrix $G$ with the property that the linear span of its rows generate the code $C$, is a \emph{generator matrix} of $C$.

Consider the projective space $PG(\mathbb{F}_q^k)$ and recall that its points are the 1-dimensional subspaces, its lines are the 2-dimensional subspaces and its hyperplanes are the $(k-1)$-dimensional subspaces of $\mathbb{F}_q^k$.  We use the abbreviation $\gaussm{k}{1}{q}=\frac{q^k-1}{q-1}$ for the number of points in $PG(\mathbb{F}_q^k)$.  The number of hyperplanes is also given by $\gaussm{k}{1}{q}$.   

Let $G_i$, $1\leq i\leq n$, be the $i$th column of a generator matrix $G$ of $C$.  To each $[n,k]_q$ code $C$, we can assign a multiset $\mathcal{P}$ of points in $PG(\mathbb{F}_q^k)$ by considering $\langle G_i\rangle$, the span of $G_i$.  For convenience of notation, we let $\mathcal{P}=\{\langle G_1\rangle,\langle G_2\rangle,\ldots,\langle G_n\rangle\}$.  Technically, a multiset of points can be described by a characteristic function $\chi$ mapping each point of $PG(\mathbb{F}_q^k)$ to a non-negative integer.  With this, the cardinality $|\mathcal{P}|$ is just the sum over $\chi(P)$ for all points $P$. By construction, $|\mathcal{P}|$ equals the effective length of $C$.  

Each non-zero codeword $c\in C$ corresponds to a hyperplane $H$ in $PG(\mathbb{F}_q^k)$ such that the set of zero coordinates of $c$ corresponds to $\mathcal{P}\cap H$.  In other words, $i\in\supp(c)$ if and only if $G_i\in PG(\mathbb{F}_q^k)\setminus (\mathcal{P}\cap H)$.  Hence, $\wt(c)=|\mathcal{P}|-|\mathcal{P}\cap H|$. We call two codewords \emph{equivalent} if they arise by a multiplication with a nonzero field element, so that equivalent codewords correspond to the same hyperplane.  

A codeword of $C\backslash\{\mathbf{0}\}$ is called \emph{minimal} if its support does not properly contain the support of another nonzero codeword.  General properties of minimal codewords are discussed in \cite{AB}.  We denote by $M(C)$ the number of non-equivalent minimal codewords in $C$, so that $M(C)\le \gaussm{k}{1}{q}$. If $G$ is a generator matrix of $C$ and $C'$ is the code that arises 
if we remove all zero-columns and all duplicated columns from $G$, then $M(C)=M(C')$. A code without zero- and duplicated columns in a generator matrix is called \emph{projective}. In geometric terms this means that the multiset $\mathcal{P}$ is indeed a set.  

We denote by $m_q(n,k)$ the minimum of $M(C)$ for all projective $[n,k]_q$ codes $C$ so that $m_q(n,k)$ is undefined if $n<k$   
or $n>\gaussm{k}{1}{q}$.  Obviously, we have $m_q(k,k)=k$, $m_q\left(\gaussm{k}{1}{q},k\right)=\gaussm{k}{1}{q}$, and $m_q(n,k)\le m_q(n',k)$ for $k\le n\le n'\le\gaussm{k}{1}{q}$.  

Similarly, we define $M_q(n,k)$ to be the maximum of $M(C)$ for all projective $[n,k]_q$ codes $C$.  This quantity was studied in \cite{Maxmin,Maxmin2} for the case of binary codes.  The focus of this work is on $m_q(n,k)$ and it is interesting to note that finding the minimum of $M(C)$ is one of the problems raised in \cite{Hwang}, the paper that introduced the concept of minimal codewords.

Kashyap showed that $m_2(n,k)\geq n$ and that the only binary codes that meet this bound are the direct sum of Simplex codes \cite{Kashyap}.
An alternative proof of the aforementioned lower bound was given in \cite{Minmin}.  The authors in \cite{Minmin} also showed that $m_2(n,n-1)=n, m_2(n,n-2)=n$ for $n\geq 6$, and computed 
bounds or exact values of $m_2(n,k)$ for $1\leq k\leq n\leq 15$.  They also determined the exact values of $m_2(n,k)$ restricted to the cycle codes from graphs for $1\leq k\leq n\leq 15$.

\section{A geometric approach to minimal codewords}

Let $C$ be a projective $[n,k]_q$ code and let $\mathcal{P}$ be the corresponding set of points in $PG(\mathbb{F}_q^k)$. For a codeword $c\in C$, we denote by $H_c$ the corresponding hyperplane in $PG(\mathbb{F}_q^k)$.  Suppose $c$ is not minimal.  Then there exists a non-zero codeword $c'$ such that $\supp(c')\subset \supp(c)$.  Equivalently, $(\mathcal{P}\cap H_c)\subset (\mathcal{P}\cap H_{c'})$.  Thus, we have the following geometric characterization of minimal codewords:  
\begin{lemma}
\label{lemma_characterization}
A non-zero codeword $c$ in an $[n,k]_q$ code $C$ is minimal if and only if $\langle \mathcal{P}\cap H_c\rangle=H_c$ or, equivalently, $\dim(\langle \mathcal{P}\cap H_c\rangle)=k-1$. 
\end{lemma}

We note that an equivalent characterization in terms of the generator matrix was obtained by Agrell \cite{Agrell2}.  We can deduce from Lemma~\ref{lemma_characterization} that if $c\in C$ is a minimal codeword then $d\leq \wt(c)\leq n-k+1$ where $d$ is the minimum Hamming weight of $C$.  This is a known property of minimal codewords, see \cite{Hwang}.

Another well-known result that can be obtained from Lemma \ref{lemma_characterization} concerns $M_q(n,k)$.  Since a $(k-1)$-subset of $\mathcal{P}$ spans a hyperplane then we have $M_q(n,k)\leq {n\choose k-1}$.  This result was first proved in \cite{DSL} for matroids, and an alternative proof was given in \cite{Maxmin2} for binary codes.
We have equality if and only if each $(k-1)$-subset of $\mathcal{P}$ spans a distinct hyperplane.  This means that $\mathcal{P}$ is an $n$-arc in $PG(\mathbb{F}_q^k)$ or, equivalently, $C$ is an MDS code.

It follows that for each non-zero non-minimal codeword $c$, there exists a subspace $U_c\le H_c$ of dimension $k-2$, i.e., co-dimension 2, with 
$\langle \left\{x\,:\,x\in\mathcal{P}\cap H_c\right\}\rangle\le U_c$. Note that there may be several such subspaces $U_c$ and the existence of at least one such subspace $U_c$ implies that $c$ is a non-minimal codeword.

We now present a lower bound on $M(C)$, the number of non-equivalent minimal codewords in $C$.  We recall that $M(C)\le \gaussm{k}{1}{q}$.  Let $\alpha_q(k,r)$ denote the minimum cardinality of a point set $\mathcal{S}\subseteq PG(\mathbb{F}_q^k)$ such that there exist $r$ different hyperplanes 
  $H_1,\dots,H_r$ and $r$ subspaces $U_1,\dots,U_r$ of co-dimension $2$ with $U_i\le H_i$ for all $1\le i\le r$ and $\cup_{i=1}^r \left(H_i\backslash U_i\right) 
  \subseteq \mathcal{S}$.  For $k=2$, we define $\alpha_q(2,r)=r$ and for $r=0$, we define $\alpha_q(k,0)=0$.

\begin{prop}
  \label{prop_bound_1}
  Let $C$ be a projective $[n,k]_q$ code and $1\le r\le\gaussm{k}{1}{q}$ be an integer. If $n>\gaussm{k}{1}{q}-\alpha_q(k,r)$ then $M(C)>\gaussm{k}{1}{q}-r$.
\end{prop} 
\begin{proof}
  If $M(C)\le \gaussm{k}{1}{q}-r$, then $C$ contains at least $r$ non-minimal codewords. These imply the existence of $r$ different hyperplanes $H_1,\dots,H_r$ and 
  $r$ subspaces $U_1,\dots,U_r$ of co-dimension $2$ with $U_i\le H_i$ for all $1\le i\le r$ and $\mathcal{P}\cap \left(\cup_{i=1}^r \left(H_i\backslash U_i\right)\right)=\emptyset$. 
  Thus, $n=|\mathcal{P}|\le \gaussm{k}{1}{q}-\alpha_q(k,r)$.   
\end{proof}

The values of $\alpha_q(k,r)$ are easy to determine analytically if $r$ is small.  First, we have $\alpha_q(k,1)=q^{k-2}$ since $|H\backslash U|=q^{k-2}$ for any hyperplane $H$ and subspace $U\leq H$ of co-dimension 2. 

\begin{prop}\label{alpha_r2}
$\alpha_q(k,2)=2q^{k-2}-q^{k-3}$ for $k\ge 3$. 
\end{prop}
\begin{proof}
We consider $\mathcal{S}=\left(H_1\backslash U_1\right)\cup\left(H_2\backslash U_2\right)$ for two distinct hyperplanes $H_1$ and $H_2$, so that $\dim(H_1\cap H_2)=k-2$.  We have $|\mathcal{S}|=2q^{k-2}-|(H_1\backslash U_1)\cap(H_2\backslash U_2)|$.  If $H_1\cap H_2=U_1$ or $H_1\cap H_2=U_2$ then $|\mathcal{S}|=2q^{k-2}$.  Otherwise, we have $|(H_1\backslash U_1)\cap(H_2\backslash U_2)|=q^{k-3}$ or $q^{k-3}-q^{k-4}$ (if $k\geq 4$).  Therefore, $\alpha_q(k,2)=2q^{k-2}-q^{k-3}$ for $k\ge 3$.
\end{proof}

A \emph{$\kappa$-arc} in $PG(\mathbb{F}_q^3)$ is a set of $\kappa$ points in $PG(\mathbb{F}_q^3)$ no three of which are collinear.  A \emph{dual $\kappa$-arc} in $PG(\mathbb{F}_q^3)$ 
is a set of $\kappa$ lines in $PG(\mathbb{F}_q^3)$ no three of which have a common point. The maximum possible $\kappa$ such that a $\kappa$-arc in $PG(\mathbb{F}_q^3)$ exists is 
well known. It is $q+2$ if the field size $q$ is even and $q+1$ otherwise, see e.g.~\cite{JH}.

\begin{prop}\label{alpha_r}
Let $r\geq 3$ and $k\geq 3$. We have $\alpha_q(k,r)=r\cdot q^{k-2}-{r\choose 2}\cdot q^{k-3}$ if $q$ is odd and $r\le q$ or if $q$ is even and $r\le q+1$.
\end{prop}
\begin{proof}
First we note that $\alpha_q(k,r)\ge r\cdot q^{k-2}-{r\choose 2}q^{k-3}$ for $k\ge 3$ and $r\ge 1$, see the analysis in the proof of Proposition \ref{alpha_r2}.  We will show that this lower bound is also tight if $r$ is not too large.
  
Fix a subspace $X$ of co-dimension 3.  All subspaces $H_i$ and $U_i$, $i=1,2,3$, to be constructed will contain $X$, thus we can describe the setting in the quotient space 
  $\overline{V}:=\mathbb{F}_q^k/X \cong \mathbb{F}_q^3$, which may be considered geometrically as a projective plane. In $\overline{V}$ we choose dual $(r+1)$-arc $L_1,\dots,L_{r+1}$, 
  which is possible due to the assumed upper bound on $r$. By construction, the intersections of the $L_i$ are pairwise disjoint. For $1\le i\le r$ let $P_i=L_i\cap L_{r+1}$, 
  i.e., the intersection point of the lines $L_i$ and $L_{r+1}$. With this, we set $H_i=\langle L_i,X\rangle$ and $U_i=\langle P_i,X\rangle$ for $1\le i\le r$.
  
Let $\mathcal{S}=\cup_{i=1}^r \left(H_i\backslash U_i\right)$. Since $\left|H_i\backslash U_i\right|=q^{k-2}$ for $1\leq i\leq r$, 
  $\left|\left(H_i\backslash U_i\right)\cap \left(H_j\backslash U_j\right)\right|=q^{k-3}$ for $1\leq i<j\leq r$, and $\cap_{i\in I} \left(H_i\backslash U_i\right)=\emptyset$ 
  (note that $\cap_{i\in I} H_i=\cap_{i\in I} U_i=X$) for all $I\subseteq\{1,\dots,r\}$ with $|I|\geq 3$, we have $|\mathcal{S}|=r\cdot q^{k-2}-{r\choose 2}\cdot q^{k-3}$.  
\end{proof}

To turn the bound of Proposition~\ref{prop_bound_1} into a statement on exact values for $m_q(n,k)$ is slightly more technical:
\begin{prop}
\label{prop_bound_2}
  For a given field size $q$, let $n$ and $k$ be positive integers with $2\le k\le n\le\gaussm{k}{1}{q}$. Let $1\le r\le\gaussm{k}{1}{q}$ be an integer 
  with $n>\gaussm{k}{1}{q}-\alpha_q(k,r)$ and $n\le \gaussm{k}{1}{q}-\alpha_q(k,r-1)$. Then $m_q(n,k)=\gaussm{k}{1}{q}-r+1$.
\end{prop}
\begin{proof}
  From Proposition~\ref{prop_bound_1} we directly conclude $m_q(n,k)\ge \gaussm{k}{1}{q}-r+1$. Let $\mathcal{S}$ be a set of points in $PG(\mathbb{F}_q^k)$ attaining 
  $\alpha_q(k,r-1)$ and $C$ be the linear code corresponding to the complement of $\mathcal{S}$. Then, $C$ has effective length $n'=\gaussm{k}{1}{q}-\alpha_q(k,r-1)\ge n$ and 
  at least $r-1$ non-minimal codewords. If $C$ has at least $r$ non-minimal codewords, then $\alpha_q(k,r)\le \alpha_q(k,r-1)$, i.e., $\alpha_q(k,r)=\alpha_q(k,r-1)$, which is 
  impossible due to our assumption on $n$. Thus, $C$ has exactly $r-1$ non-minimal codewords. Since $n'\ge n$ we have $m_q(n,k)\le m_q(n',k) \le \gaussm{k}{1}{q}-r+1$.    
\end{proof}  


Setting $r=1$ in Proposition \ref{prop_bound_2}, we obtain the following: for $k\geq 2$, if $\gaussm{k}{1}{q}-q^{k-2}<n\leq \gaussm{k}{1}{q}$ then $m_q(n,k)=\gaussm{k}{1}{q}$ .  Next we show that $m_q\left(\gaussm{k}{1}{q}-q^{k-2},k\right)<\gaussm{k}{1}{q}$.  Let $H$ be a hyperplane and $U\leq H$ a subspace of co-dimension 2.  Consider the code $C$ whose point set $\mathcal{P}=PG(\mathbb{F}_q^k)\backslash (H\backslash U)$.  Note that $|\mathcal{P}|=\gaussm{k}{1}{q}-q^{k-2}$.  Then $C$ has at least one non-minimal codeword (the one associated with $H$).  

Since $m_q(n,k)$ attains the maximum possible value for $M(C)$ then all codes in this range have the property that all non-zero codewords are minimal.  These codes are called \emph{minimal codes} and were first studied in \cite{AB,DY}.  Minimal codes were also used in the protocol for secure two-party computation proposed in \cite{CCP}.  If $C$ is an $[n,k]_q$ minimal code then it was shown in \cite{LWC,TQLZ,ABN} that the length satisfies $n\geq (k-1)q+1$.  The case of $r=1$ above gives a tight lower bound for projective $[n,k]_q$ minimal codes as $n\geq \gaussm{k}{1}{q}-q^{k-2}+1$.

When $r=2$ in Proposition \ref{prop_bound_2}, we get: for $k\geq 3$, if $\gaussm{k}{1}{q}-2q^{k-2}+q^{k-3}<n\leq \gaussm{k}{1}{q}-q^{k-2}$ then $m_q(n,k)=\gaussm{k}{1}{q}-1$.  For this range of $k$ and $n$, the value of $m_q(n,k)$ is the maximum possible value.  Hence, we can say that each code $C$ in this range has $M(C)=\gaussm{k}{1}{q}-1$, i.e. has exactly one non-minimal codeword.

We can apply the above discussion to update the tables given in \cite{Minmin}.  For example, we have $m_2(6,3)=7$ and $m_2(n,4)=15$ for $n=12,13,14,15$.  For the remaining entries of 
Table~\ref{tab_minimal_codewords} we consider an exhaustive enumeration of linear codes. First note that if a linear code $C$ contains a codeword of weight 1 then removing the 
corresponding coordinate yields a code $C'$ with $n(C') = n(C)-1$ and $M(C') = M(C)-1$. Thus it is sufficient to consider all projective $[n,k]_2$ codes with minimum distance at least 2. 
These can be generated easily and for each code we can simply count the number of minimal codewords. To this end we have applied the algorithm from \cite{LC}. 

\begin{table}[htbp]
\begin{center}
{\small
\begin{tabular}{|c|c|c|c|c|c|c|c|c|c|c|c|c|c|c|}\hline
$n/k$ & 2 & 3   & 4     & 5     & 6     & 7     & 8     & 9  & 10 & 11 & 12 & 13 & 14 & 15 \\\hline
3     & 3 & 3   &       &       &       &       &       &    &    &    &    &    &    & \\\hline
4     &   & 4   & 4     &       &       &       &       &    &    &    &    &    &    & \\\hline
5     &   & 6   & 5     & 5     &       &       &       &    &    &    &    &    &    & \\\hline
6     &   & 7   & 6     & 6     & 6     &       &       &    &    &    &    &    &    & \\\hline
7     &   & 7   & 8     & 7     & 7     & 7     &       &    &    &    &    &    &    & \\\hline
8     &   &     & 8     & 9     & 8     & 8     & 8     &    &    &    &    &    &    & \\\hline
9     &   &     & 12    & 9     & 9     & 9     & 9     & 9 &    &    &    &    &    & \\\hline
10    &   &     & 14    & 10    & 10    & 10    & 10    & 10 & 10 &    &    &    &    & \\\hline
11    &   &     & 14    & 15    & 11    & 11    & 11    & 11 & 11 & 11 &    &    &    & \\\hline
12    &   &     & 15    & 15    & 13    & 12    & 12    & 12 & 12 & 12 & 12 &    &    & \\\hline
13    &   &     & 15    & 16    & 14    & 13    & 13    & 13 & 13 & 13 & 13 & 13 &    & \\\hline
14    &   &     & 15    & 16    & 14    & 15    & 14    & 14 & 14 & 14 & 14 & 14 & 14 & \\\hline
15    &   &     & 15    & 16    & 17    & 15    & 16    & 15 & 15 & 15 & 15 & 15 & 15 & 15\\
\hline\end{tabular}}
\end{center}
\caption{$m_2(n,k)$ for $3\leq n\leq 15, 1\leq k\leq 9$}
\label{tab_minimal_codewords}
\end{table}

\section{Minimal subcode supports}

The geometric approach used in the previous section can be extended to subcode supports.  Let $C$ be a projective $[n,k]_q$ code and let $D$ be an $l$-dimensional subcode of $C$.  The \emph{support of D}, denoted by $\supp(D)$, is the union of the supports of all the codewords in $D$ and the \emph{weight of D}, denoted by $\wt(D)$, is the cardinality of its support.  The $l$-th \emph{generalized Hamming weight} $d_l$ of $C$ is the minimum among the weights of the $r$-dimensional subcodes of $C$ \cite{Wei}.  In short,
\begin{align*}
\supp(D)&=\{i\in\{1,\ldots,n\}\;:\;\exists \;v\in D \text{ with } v_i\neq 0\}\\
\wt(D)&=|\supp(D)|\\
d_l&=\min\{\wt(D)\;:\;D\leq C, \dim(D)=l\}.
\end{align*}

For a given subcode $D$ with $\dim(D)=l$, we can associate a subspace in $PG(\mathbb{F}_q^k)$ of codimension $l$.  Let $G$ be a generator matrix for $C$.  Then there exists an $l\times k$ matrix $M$ such that the rows of $MG$ form a basis for $D$.  The nullspace $W$ of $M$ is a subspace in $PG(\mathbb{F}_q^k)$ of co-dimension $l$.  In fact, there is a one-to-one correspondence between the subcodes of $C$ of dimension $l$ and subspaces of $PG(\mathbb{F}_q^k)$ of co-dimension $l$ (for more details, see \cite{TV,Jurrius}).
 
Let $\mathcal{P}\subseteq PG(\mathbb{F}_q^k)$ be the set of points associated with $C$. Let $D$ be an $l$-dimensional subcode of $C$. Then $D$ corresponds to a subspace $W$ in $PG(\mathbb{F}_q^k)$ of co-dimension $l$.  From \cite{Jurrius}, we have $\supp(D)=PG(\mathbb{F}_q^k)\backslash(\mathcal{P}\cap W)$ and $wt(D)=n-|\mathcal{P}\cap W|$. The $l$-th generalized Hamming weight $d_l=n-\min\{|\mathcal{P}\cap W|\;:\; W \text{ subspace of co-dimension } l\}$.  We say that $D$ is a \emph{support-minimal subcode} if there is no other $l$-dimensional subcode $D'\leq C$ such that $\supp(D')\subset \supp(D)$.  

The following lemma extends the geometric characterization in the previous section to subcodes:
\begin{lemma}\label{min_subcode}
Let $C$ be a projective $[n,k]_q$ code and $\mathcal{P}$ be the corresponding set of points in $PG(\mathbb{F}_q^k)$.  Let $D$ be an $l$-dimensional subcode of $C$ and consider the associated subspace $W_D$ in $PG(\mathbb{F}_q^k)$ of co-dimension $l$. Then $supp(D)$ is minimal if and only if $\langle \mathcal{P}\cap W_D\rangle=W_D$.  Equivalently, $\dim (\langle \mathcal{P}\cap W_D\rangle) = k-l$.
\end{lemma}

If $D$ is a support-minimal subcode with $\dim(D)=l$ then $d_l\leq \wt(D)\leq n-k-l$, where $d_l$ is the $l$-th generalized Hamming weight of $C$.  Minimal subcode supports were studied as circuits of certain matroids in \cite{Britz}.  It was shown that the set of minimal subcode supports determines the multiset of subcode supports.  For $1\leq l'\leq l\leq k$, the set of minimal $l'$-dimensional subcode supports also determines the set of minimal $l$-dimensional subcode supports.   

\begin{example}
We look at some codes and their support-minimal subcodes.
\begin{enumerate}
\item Simplex codes.  Let $C$ be the $k$-th order $q$-ary Simplex code which has parameters $[(q^k-1)/(q-1),k,q^{k-1}]$.  The columns of the generator matrix for $C$ form a set of non-zero representatives of the 1-dimensional subspaces of $\mathbb{F}_q^k$.  This means that the point set associated with $C$ is $PG(\mathbb{F}_q^k)$.  By Lemma \ref{min_subcode}, for a given $1\leq l\leq k$, all the $l$-dimensional subcodes of $C$ are support-minimal and have the same weight.
 
\item $l$-MDS codes.  Let $C$ be an $l$-MDS code, i.e. $d_l=n-k+l$.  It follows that among the $l$-dimensional subcodes of $C$, the only support-minimal subcodes are those with weight equal to $d_l$.  For an $l$-MDS code, we have $d_l'=n-k+l'$ for $l'\geq l$.  Hence, among the $l'$-dimensional subcodes of $C$, the only support-minimal subcodes are those with weight equal to $d_l'$.  In particular, if $C$ is an MDS code then we can completely determine all the support-minimal subcodes for $1\leq l\leq k$.
\end{enumerate} 
\end{example}

For $1\leq l\leq k$, we define $M^l(C)$ to be the number of support-minimal $l$-dimensional subcodes of $C$.  When $l=1$ we get $M^1(C)=(q-1)M(C)$.  An upper bound for $M^l(C)$ is given by $M^l(C)\leq \gaussm{k}{l}{q}$ where $$\gaussm{k}{l}{q} = \dfrac{(q^k-1)(q^{k-1}-1)\cdots(q^{k-l+1}-1)}{(q^l-1)(q^{l-1}-1)\cdots(q-1)}$$ is the Gaussian binomial coefficient that gives the number of subspaces in $PG(\mathbb{F}_q^k)$ of co-dimension $l$.  

From Lemma \ref{min_subcode}, a subcode $D$ is not support-minimal if there exists a subspace $U_D\leq W_D$ of co-dimension $l+1$ such that $\langle \mathcal{P}\cap W_D\rangle \leq U_D$.  For $1\leq l\leq k$ and $1\leq r\leq \gaussm{k}{l}{q}$, we define $\alpha^l_q(k,r)$ to be the minimum cardinality of a point set $\mathcal{S}\subseteq PG(\mathbb{F}_q^k)$ such that there exist $r$ distinct subspaces $W_1,\ldots,W_r$ of co-dimension $l$ and $r$ subspaces $U_1,\ldots,U_r$ of co-dimension $l+1$ with $U_i\leq W_i$ and $\cup_{i=1}^r \left(W_i\backslash U_i\right) \subseteq \mathcal{S}$.  The next proposition extends Proposition \ref{prop_bound_1} to subcodes.

\begin{prop}\label{prop_bound_3}
Let $C$ be a projective $[n,k]_q$ code and consider integers $l,r$ such that $1\leq l\leq k$ and $1\le r\le\gaussm{k}{l}{q}$. If $n>\gaussm{k}{l}{q}-\alpha_q^l(k,r)$ then $M^l(C)>\gaussm{k}{l}{q}-r$. 
\end{prop}
\begin{proof}
The proof is similar to \ref{prop_bound_1}.
\end{proof}

\section{Concluding remarks}

We presented a geometric characterization of minimal codewords in linear codes.  We then applied this characterization to the problem of finding the minimum number of minimal codewords of projective $[n,k]_q$ codes.  We obtain a new lower bound and in some cases, we were able to determine exact values that were not known before.  We also extended the techniques to the study of subcodes with minimal supports. 

\section*{Acknowledgments}

R. dela Cruz gratefully acknowledges the support of the Alexander von Humboldt Foundation and the Department of Mathematics, University of Bayreuth.


\medskip
\begin{thebibliography}{99}

\bibitem{Agrell} E. Agrell. Voronoi Regions for Binary Linear Block Codes. \textit{IEEE Trans. Inf. Theory}, vol. 42, no. 1, pp. 310-316, 1998.

\bibitem{Agrell2} E. Agrell. On the Voronoi neighbor ratio for binary linear codes. \textit{IEEE Trans. Inf. Theory}, vol. 44, no. 7, pp. 3064-3072, 1998.

\bibitem{Minmin} A. Alahmadi, R.E.L. Aldred, R. dela Cruz, S. Ok, P. Sol\'e and C. Thomassen. The minimum number of minimal codewords in an $[n,k]$-code. \textit{Discrete Applied Mathematics}, vol. 184, pp. 32-39, 2015.

\bibitem{Maxmin2} A. Alahmadi, R.E.L. Aldred, R. dela Cruz, P. Sol\'e and C. Thomassen. The maximum number of minimal codewords in an $[n,k]$-code. \textit{Discrete Mathematics}, vol. 313, issue 15, pp. 1569-1574, 2013.

\bibitem{Maxmin} A. Alahmadi, R.E.L. Aldred, R. dela Cruz, P. Sol\'e and C. Thomassen. The maximum number of minimal codewords in long codes. \textit{Discrete Applied Mathematics}, vol. 161, issue 3, pp. 424-429, 2013.

\bibitem{ABN} G. N. Alfarano M. Borello and A. Neri. A geometric characterization of minimal codes and their asymptotic performance. arXiv preprint arXiv:1911.11738, 2019.

\bibitem{AB} A. Ashikhmin and A. Barg. Minimal vectors in linear codes. \textit{IEEE Trans. Inf. Theory}, vol. 44, no. 5, pp. 2010-2017, 1998.

\bibitem{BM} Y. Borissov and N. Manev. Minimal codewords in linear codes. \textit{Serdica Mathematical Journal}, vol. 30, pp. 303-324, 2004.


\bibitem{Britz} T. Britz. Higher support matroids. \textit{Discrete Mathematics}, vol. 307, issue 17–18, pp. 2300-2308, 2007.

\bibitem{CCP} H. Chabanne, G. Cohen and A. Patey. Towards Secure Two-Party Computation from the Wire-Tap Channel. In \textit{Proc. Information Security and Cryptology ICISC 2013}, LNCS, vol. 8565, pp. 34-46.

\bibitem{DKL} C. Ding, D. Kohel and S. Ling. Secret-sharing with a class of ternary codes. \textit{Theoretical Computer Science}, vol. 246, issues 1-2, pp. 285-298, 2000.

\bibitem{DY} C. Ding and J. Yuan. Covering and secret sharing with linear codes. In \textit{Proc. 4th Int. Conf. on Discrete Mathematics and Theoretical Computer Science}, Dijon, France, pp. 11-25, 2003.  

\bibitem{DSL} G. Y. Dosa, I. Szalkai and C. Laflamme. The maximum and minimum number of circuits and bases of matroids. \textit{PU.M.A.}, vol. 15, no. 4, pp. 383-392, 2004.

\bibitem{ES} R. Entringer and P. Slater. On the maximum number of cycles in a graph. \textit{Ars Combinatoria}, vol. 11, pp. 289-294, 1981.

\bibitem{JH} J. W. P. Hirschfeld and L. Storme. The packing problem in statistics, coding theory and finite projective spaces. \textit{J. Statist. Planning Infer.}, vol. 72, pp. 355-380, 1998. 

\bibitem{Hwang} T.-Y. Hwang. Decoding linear block codes for minimizing word error rate. \textit{IEEE Trans. Inf. Theory}, vol. IT-25, pp. 733-737, 1979. 

\bibitem{Jurrius} R. Jurrius. Weight enumeration of codes from finite spaces. \textit{Des. Codes Cryptogr.}, vol. 63, issue 3, pp. 321-330, 2012.

\bibitem{Kashyap} N. Kashyap. On the convex geometry of binary linear codes. preprint. http://ita.ucsd.edu/workshop/06/papers/82.pdf.

\bibitem{LC} S. Kurz. LinCode - computer classification of linear codes. arXiv preprint 1912.09357, 2019.

\bibitem{LWC} W. Lu X. Wu and X. Cao. The Parameters of Minimal Linear Codes. arXiv preprint 1911.07648, 2019.

\bibitem{Massey} J. L. Massey. Minimal codewords and secret sharing. In \textit{Proc. 6th Joint Swedish-Russian Workshop Inf. Theory}, Molle, Sweden, pp. 276-279, 1993.   

\bibitem{SST} J. Schillewaert, L. Storme and J. A. Thas. Minimal codewords in Reed-Muller codes. \textit{Designs, Codes and Cryptography}, vol. 54, issue 3, pp. 273-286, 2010.

\bibitem{TQLZ} C. Tang, Y. Qiu, Q. Liao, and Z. Zhou. Full characterization of minimal linear codes as cutting blocking sets. arXiv preprint 1911.09867, 2019.


\bibitem{TV} M. Tsfasman and S. Vladut. Geometric approach to higher weights. \textit{IEEE Trans. Inf. Theory}, vol. 41, issue 6, pp. 1564-1588, 1995.

\bibitem{Wei} V. Wei. Generalized Hamming weights for linear codes. \textit{IEEE Trans. Inf. Theory}, vol. 37, issue 5, pp. 1412-1418, 1991.

\bibitem{YF} K. Yasunaga and T. Fujiwara. Determination of the Local Weight Distribution of Binary Linear Block Codes. \textit{IEEE Trans. Inf. Theory}, vol. 52, issue 10, pp. 4444-4454, 2006.

\end{thebibliography}
\end{document}